\documentclass[12pt]{article}
\usepackage{amsfonts,amsmath,amscd,amssymb}
\usepackage{graphicx}
\usepackage{color}

\newif\ifblog
\newif\iftex
\blogfalse
\textrue

\newcommand{\nd}{\noindent}

\newcommand{\bed}{\begin{displaymath}}
\newcommand{\eed}{\end{displaymath}}
\newcommand{\bea}{\bed\begin{array}{rl}}
\newcommand{\eea}{\end{array}\eed}

\newcommand{\barray}{\begin{array}{ll}}
\newcommand{\earray}{\end{array}}

\newtheorem{theorem}{Theorem}[section]
\newtheorem{lemma}[theorem]{Lemma}

\newtheorem{remark}[theorem]{Remark}

\newenvironment{proof}{\noindent {\sc Proof:}}{\strut\hfill $\Box$} %\medskip}

\def\emr#1{\emph{\bf{{\color{red}#1}}}}

\title{Optimal Investment Stopping Problem\\ with Nonsmooth Utility in Finite Horizon}%\thanks{The project is supported by
\author{Chonghu Guan,\thanks{School of Mathematics, Jiaying University, Meizhou 514015, China}\;\;
  Xun Li,\thanks{Department of Applied Mathematics, Hong Kong Polytechnic University, Hong Kong} \;\;
   Zuoquan Xu\thanks{Department of Applied Mathematics, Hong Kong Polytechnic University, Hong Kong}\;  and\;
   Fahuai Yi\thanks{School of Mathematical Science, South China Normal University, Guangzhou 510631, China.
The project is supported by
 NNSF of China (No.11271143, No.11371155 and No.11471276), University Special Research
 Fund for Ph.D. Program of China (20124407110001), and Research Grants Council of Hong Kong under grants 521610 and 519913.}
}
\date{}

\begin{document}

\maketitle
\vspace{-1cm}
\begin{abstract}
In this paper, we investigate an interesting and important stopping problem mixed with stochastic controls and a \textit{nonsmooth} utility over a finite time horizon. The paper aims to develop new methodologies, which are significantly different from those of mixed dynamic optimal control and stopping problems in the existing literature, to figure out a manager's decision. We formulate our model to a free boundary problem of a fully \textit{nonlinear} equation. By means of a dual transformation, however, we can convert the above problem to a new free boundary problem of a \textit{linear} equation. Finally, using the corresponding inverse dual transformation, we apply the theoretical results established for the new free boundary problem to obtain the properties of the optimal strategy and the optimal stopping time to achieve a certain level for the original problem over a finite time investment horizon.

\bigskip
\nd {\bf Keywords}: Parabolic variational inequality; Free boundary; Optimal investment; Optimal stopping; Dual transformation.

\bigskip
\nd {\bf Mathematics Subject Classification.} 35R35; 60G40; 91B70; 93E20.

%\bigskip
%\nd {\bf Brief Title.} Dividend and Risk Control Model

\end{abstract}

%\newpage

\setlength{\baselineskip}{0.25in}
\section{Introduction}\setcounter{equation}{0}
Optimal stopping problems have important applications in many fields such as science,
engineering, economics and, particularly, finance. The theory in this area
has been well developed for stochastic dynamic systems over the past decades.
In the field of financial
investment, however, an
investor frequently runs into investment decisions where investors
stop investing in risky assets so as to maximize their expected
utilities with respect to their wealth over a finite time investment
horizon. These optimal stopping problems depend on underlying
dynamic systems as well as investors' optimization decisions (controls).
This naturally results in a mixed optimal control and stopping problem, and
Ceci-Bassan (2004) is one of the typical representatives along this line of
research. In the general formulation of such models, the control is mixed,
composed of a control and a stopping time. The theory has also been
studied in Bensoussan-Lions (1984), Elliott-Kopp (1999), Yong-Zhou (1999) and Fleming-Soner (2006),
and applied in finance in
Dayanik-Karatzas (2003), Henderson-Hobson (2008), Li-Zhou (2006), Li-Wu
(2008, 2009), Shiryaev-Xu-Zhou (2008) and Jian-Li-Yi (2014).

In the finance field, finding an optimal stopping time point has been
extensively studied for pricing American-style options, which allow option
holders to exercise the options before or at the maturity. Typical examples
that are applicable include, but are not limited to, those presented in
Chang-Pang-Yong (2009), Dayanik-Karatzas (2003) and R\"uschendorf-Urusov
(2008). In the mathematical finance literature, choosing an optimal stopping
time point is often related to a free boundary problem for a class of
diffusions (see Fleming-Soner (2006) and Peskir-Shiryaev (2006)).
In many applied areas, especially in more extensive investment
problems, however, one often encounters more general controlled diffusion
processes. In real financial markets, the situation is even more complicated
when investors expect to choose as little time as possible to stop portfolio
selection over a given investment horizon so as to maximize their profits
(see Samuelson (1965), Karatzas-Kou (1998), Karatzas-Sudderth (1999),
Karatzas-Wang (2000), Karatzas-Ocone (2002), Ceci-Bassan (2004), Henderson (2007),
Li-Zhou (2006) and Li-Wu (2008, 2009)).

The initial motivation of this paper comes from our recent studies
on choosing an optimal point at which an investor stops investing
and/or sells all his risky assets (see Carpenter (2000) and Henderson-Hobson (2008)). The objective is to find an
optimization process and stopping time so as to meet certain
investment criteria, such as, the maximum of an expected nonsmooth utility
value before or at the maturity. This is a typical yet important problem in the
area of financial investment. However, there are fundamental
difficulties in handling such mixed controls and stopping problems. Firstly, our
investment problem, which is signifcantly different from the classical
American-style options, involves portfolio process in the objective over the entire
time horizon. Secondly, it
involves the portfolio in the drift and volatility terms of the dynamic systems so that the
problem including multi-dimensional financial assets is more realistic
than those addressed in finance literature (see Capenter (2000)). Therefore, it is difficult to solve
these problems either analytically or numerically using current
methods developed in the framework of studying American-style options.
In our model, the corresponding HJB equation of the problem is formulated into a variational inequality
of a fully \textit{nonlinear} equation. We make a dual transformation for the
problem to obtain a new free boundary problem with a \textit{linear} equation. Tackling this new free boundary problem,
we characterize the properties of the free boundary and optimal
strategy for the original problem.

The main innovations of this paper include that: Firstly, we rigorously prove the limit of the value function when $t\rightarrow T$ is the concave hull of the payoff function (but not the payoff function itself), i.e.
\[
\lim\limits_{t\rightarrow T-}V(x,t)=\varphi(x),
\]
where $\varphi(x)$ is the concave hull of the payoff function $g(x)$ (see Lemma \ref{lem:Terminal}).
%(which is owing to the HJB equation is singular with the portfolio $\pi_t$ is unbounded.)
Secondly, since the obstacle $\varphi(x)$ in variational inequality is not strictly concave (see Figure 2.1),
the equivalence between the dual problem \eqref{v} and the original problem \eqref{V} is not trivial.
However, we successfully proved it in Section 3.
Thirdly, we show a new method to study the free boundary while the exercise region is not connected (see \eqref{h(t)}-\eqref{f(t)} and Lemma \ref{lem:inCR})
so that we can shed light on the monotonicity and differentiability of the free boundaries (see Figure 5.1-5.4.) under any cases of parameters.

The remainder of the paper is organized as follows. In Section 2, the mathematical formulation of the
model is presented, and the corresponding HJB equation with certain boundary-terminal condition is posed.
In particular, we show that the value function $V(x,t)$ is not continuous at $t=T$, i.e.
\[
\lim\limits_{t\rightarrow T-}V(x,t)\neq V(x,T).
\]
In Section 3, we make a dual transformation to convert the free boundary problem of a fully \textit{nonlinear} PDE \eqref{V} to
a new free boundary problem of a \textit{linear} equation \eqref{v}.
Section 4 devotes to the study for the free boundary of problem \eqref{v} in different cases of parameters. In Section 5, using the corresponding inverse dual transformation, we construct the solution of the original
problem \eqref{V} and to present the properties (including the monotonicity and differentiability) are its free boundaries under different cases which is classified in Section 4.
In Section 6 we present conclusions.
%The innovations of this paper include that: {\bf Firstly}, we rigorously prove the limit of the value function when $t\rightarrow T$ is the concave hull of the payoff function $g(x)$ but not the payoff function itself, i.e.
%\[
%\lim\limits_{t\rightarrow T-}V(x,t)=\varphi(x).
%\]
%where, $\varphi(x)$ is the concave hull of the payoff function $g(x)$(see Lemma \ref{lem:Terminal}).
%(which is owing to the HJB equation is singular with the portfolio $\pi_t$ is unbounded.) {\bf Secondly}, since the obstacle $\varphi(x)$ in variational inequality is not strictly concave(see Figure 2.1), so the equivalence between the dual problem \eqref{v} and the original problem \eqref{V} is not nature. Even so, we still proved it in Section 3. {\bf Thirdly}, we show a new method to research the free boundary while the exercise region is not connected(see \eqref{h(t)}-\eqref{f(t)} and Lemma \ref{lem:inCR}), such that under any cases of parameters, the monotonicity and differentiability of the free boundaries are presented by us. (see Figure 5.1-5.4.)

%%%%%%%%%%%%%%%%%%%%%%%

%\section{The HJB equation}\label{sec:The HJB equation}\setcounter{equation}{0}

\section{Model Formulation}
\setcounter{section}{2} \setcounter{equation}{0}

\subsection{The manager's problem}
The manager operates in a complete, arbitrage-free, continuous-time
financial market consisting of a riskless asset with instantaneous
interest rate $r$ and $n$ risky assets. The risky asset prices $S_i$ are
governed by the stochastic differential equations
\begin{eqnarray}\label{eq:S}
\frac{dS_{i,t}}{S_{i,t}}=(r+\mu_i)dt+\sigma_{i}dW_t^j, \quad \mbox{for } i = 1,2,\cdots,n,
\end{eqnarray}
where the interest rate $r$, the excess appreciation rates $\mu_i$, and the volatility vectors
$\sigma_i$ are constants,
$W$ is a standard $n$-dimensional Brownian motion.
In addition, the covariance matrix $\sigma\sigma'$ is strongly nondegenerate.

A trading strategy for the manager is an $n$-dimensional process $\pi_t$ whose $i$-th component, where $\pi_{i,t}$
is the holding amount of the $i$-th risky asset in the
portfolio at time $t$. An admissible trading strategy $\pi_t$
must be progressively measurable with respect to $\{\mathcal{F}_t\}$ such that $X_t\ge0$. Note that
$X_t=\pi_{0,t}+\sum\limits_{i=1}^n\pi_{i,t}$, where $\pi_{0,t}$ is the amount invested
in the money.
Hence, the wealth $X_t$ evolves according to
\[
dX_t=(rX_t+\mu'\pi_t)dt+\pi_t'\sigma dW_t,
\]
the portfolio $\pi_t$ is a progressively measurable and square integrable process with constraint $X_t\geq0$ for all $t\geq 0$.

The manager's dynamic problem is to choose an admissible trading strategy $\pi_t$ and a stopping time $\tau$ ($t\leq\tau\leq T$) to maximize his expected utility of the exercise wealth before or at the terminal time $T$:
\begin{eqnarray}\label{value}
V(x,t)=\sup\limits_{\pi,\tau}E_{t,x}[e^{-\beta(\tau-t)}g(X_\tau)]:=\sup\limits_{\pi,\tau}E[e^{-\beta(\tau-t)}g(X_\tau)|X_t=x],
\end{eqnarray}
where
\[
g(x)=\frac{1}{\gamma}[(x-b)^++K]^\gamma,
\]
and $\beta$ is the discounted factor.

If $X_t=0$, in order to keep $X_s\geq0$, the only choice of $\pi_s$ is 0 and thus $X_s\equiv0,\;t\leq s\leq T$. Thus
\[
V(0,t)=\sup\limits_{\pi,\tau}E[e^{-\beta(\tau-t)}g(0)]=g(0)=\frac{1}{\gamma}K^\gamma.
\]
Which means the optimal stopping time $\tau$ is the present moment $t$.

\subsection{Discontinuity at the terminal time $T$}

From the definition \eqref{value} we can see that $V(x,T)=g(x)=\frac{1}{\gamma}[(x-b)^++K]^\gamma$. Since the portfolio $\pi_t$ is unrestricted,  $V(x,t)$ may be discontinuous at the terminal time $T$.
Therefore, we should pay attention to $V(x,T-):=\lim\limits_{t\rightarrow T-}V(x,t)$.

\begin{lemma}\label{lem:Terminal}
\sl
The value function $V$ defined in \eqref{value} is not continuous at the terminal time $T$ and satisfies
\[
\lim\limits_{t\rightarrow T-}V(x,t)=\varphi(x),
\]
where
\[
\varphi(x)=
\left\{\begin{array}{l}
kx+\frac{1}{\gamma}K^\gamma,\quad 0<x<\widehat{x}, \\ [2mm]
\frac{1}{\gamma}(x-b+K)^\gamma, \quad x\geq \widehat{x},
\end{array}\right.
\]
is the concave hull of $\frac{1}{\gamma}[(x-b)^++K]^\gamma$ (see Fig 2.1),
here $k$ and $\widehat{x}$ satisfy
\begin{eqnarray}\label{k x}
\left\{\begin{array}{l}
k \widehat{x}+\frac{1}{\gamma}K^\gamma=\frac{1}{\gamma}(\widehat{x}-b+K)^\gamma, \\ [2mm]
k=(\widehat{x}-b+K)^{\gamma-1}.
\end{array}\right.
\end{eqnarray}

\vspace{1cm}

\begin{center}
\begin{picture}(250,100)
    \put(50,10){\vector(1,0){150}}
    \put(60,0){\vector(0,1){95}}
    \put(20,100){Utility}
    \put(80,80){$\varphi(x)$}
    \put(95,-5){$b$}\put(100,10){\circle*{2}}
    \put(120,-5){$\widehat{x}$}\put(125,10){\circle*{2}}
    \put(205,-2){$x$}
    \qbezier(60,30)(60,30)(130,75)\qbezier(100,30)(114,89)(200,95)
    \put(60,30){\line(1,0){40}}
    \put(120,35){$\frac{1}{\gamma}[(x-b)^++K]^\gamma$}
    \put(90,70){\vector(1,-1){10}}
    \put(120,40){\vector(-1,0){10}}
    \put(80,-30){Fig 2.1 \;$\varphi(x)$.}
    \put(30,25){$\frac{1}{\gamma}K^\gamma$}
  \end{picture}
\end{center}
\vspace{1cm}
\end{lemma}

\begin{proof}
We first prove
\[
\limsup\limits_{t\rightarrow T-}V(x,t)\leq\varphi(x).
\]

Define
\[
\zeta_t=e^{-(r+\mu'(\sigma'\sigma)^{-1}\mu)t-\mu'\sigma^{-1} W_t},
\]
then
\[
d\zeta_t=\zeta_t[-rdt-\mu'\sigma^{-1} dW_t]
\]
and
\begin{eqnarray}
d(\zeta_tX_t)
& = & \zeta_tdX_t+X_td\zeta_t+d\zeta_tdX_t\nonumber\\
& = & \zeta_t[(rX_t+\mu'\pi_t)dt+\pi_t'\sigma dW_t-rX_tdt-\mu'\sigma^{-1} X_tdW_t
-(\mu'\sigma^{-1})(\pi_t'\sigma)'dt]\nonumber\\\label{dX}
& = & \zeta_t[\pi_t'\sigma -\mu'\sigma^{-1} X_t] dW_t.
\end{eqnarray}
Thus, $\zeta_tX_t$ is a martingale. For any $\pi\in {\cal A}$ and stopping time $\tau$ ($t\leq\tau\leq T$), by Jensen's inequality, we have
\[
E_{t,x}\varphi\Big(\frac{\zeta_\tau}{\zeta_t} X_\tau\Big)\leq \varphi\Big(E_{t,x}\Big(\frac{\zeta_\tau}{\zeta_t} X_\tau\Big)\Big)= \varphi(x).
\]
Then
\begin{eqnarray}\label{Terminal}
\limsup\limits_{t\rightarrow T-}\sup\limits_{\tau,\;\pi}E_{t,x}\varphi\Big(\frac{\zeta_\tau}{\zeta_t} X_\tau\Big)\leq \varphi(x).
\end{eqnarray}
We now come to prove
\begin{eqnarray}\label{Terminal1}
\lim\limits_{t\rightarrow T-}\sup\limits_{\tau,\;\pi}E_{t,x}\Big|\varphi(X_\tau)-\varphi\Big(\frac{\zeta_\tau}{\zeta_t} X_\tau\Big)\Big|=0.
\end{eqnarray}

Indeed, owing to $\varphi(x)$ is differentiable and for all $x,\;y\geq\widehat{x}$,
\[
|(x-b+K)^\gamma-(y-b+K)^\gamma|\leq |x-y|^\gamma,
\]
there exits constant $C>0$ such that for all $x,\;y>0$,
\[
|\varphi(x)-\varphi(y)|\leq C|x-y|^\gamma.
\]
Thus, for any $\pi$ and stopping time $t\leq\tau\leq T$,
\begin{eqnarray*}
E_{t,x} \Big|\varphi (X_\tau )-\varphi \Big(\frac{\zeta_\tau}{\zeta_t} X_\tau \Big)\Big|\leq C E_{t,x}\Big(\Big(\frac{\zeta_\tau}{\zeta_t} X_\tau\Big)^\gamma\Big|\frac{\zeta_t}{\zeta_\tau} -1\Big|^\gamma\Big).
\end{eqnarray*}
Using H\"older inequality, we obtain
\begin{eqnarray*}
E_{t,x} \Big|\varphi (X_\tau )-\varphi \Big(\frac{\zeta_\tau}{\zeta_t} X_\tau \Big)\Big|
&\leq&
C \Big(E_{t,x} \Big(\frac{\zeta_\tau}{\zeta_t} X_\tau\Big) \Big)^\gamma\Big(E_{t,x}\Big|\frac{\zeta_t}{\zeta_\tau}-1\Big|^{\frac{\gamma}{1-\gamma}}\Big)^{1-\gamma}\\
&\leq&
C x^\gamma\Big(E_{t,x}\sup\limits_{t\leq s\leq T}\Big|\frac{\zeta_t}{\zeta_s}-1\Big|^{\frac{\gamma}{1-\gamma}}\Big)^{1-\gamma}.
\end{eqnarray*}
Hence,
\begin{eqnarray*}
\lim\limits_{t\rightarrow T-}\sup\limits_{\tau,\;\pi}E_{t,x}\Big|\varphi(X_\tau)-\varphi\Big(\frac{\zeta_\tau}{\zeta_t} X_\tau\Big)\Big|\leq C x^\gamma\lim\limits_{t\rightarrow T-}\Big(E_{t,x}\sup\limits_{t\leq s\leq T}\Big|\frac{\zeta_t}{\zeta_s}-1\Big|^{\frac{\gamma}{1-\gamma}}\Big)^{1-\gamma}=0.
\end{eqnarray*}
Therefore, by \eqref{Terminal} and \eqref{Terminal1},
\begin{eqnarray*}
\limsup\limits_{t\rightarrow T-}V(x,t)
&=&
\limsup\limits_{t\rightarrow T-}\sup\limits_{\tau,\;\pi}E_{t,x}\Big(e^{-\beta(\tau-t)}g(X_\tau)\Big)\\
&\leq&
\limsup\limits_{t\rightarrow T-}\sup\limits_{\tau,\;\pi}E_{t,x}\varphi(X_\tau)\\
&\leq&
\limsup\limits_{t\rightarrow T-}\sup\limits_{\tau,\;\pi}E_{t,x}\varphi\Big(\frac{\zeta_\tau}{\zeta_t} X_\tau\Big) +\lim\limits_{t\rightarrow T-}\sup\limits_{\tau,\;\pi}E_{t,x}\Big|\varphi(X_\tau)-\varphi\Big(\frac{\zeta_\tau}{\zeta_t} X_\tau\Big)\Big|\\
&\leq&
\varphi(x).
\end{eqnarray*}

Next, we further prove
\begin{eqnarray}\label{Terminal2}
\liminf\limits_{t\rightarrow T-}V(x,t)\geq\varphi(x).
\end{eqnarray}

For fix $t<T$, if $x\geq x_0$ or $x=0$, we can get
\[
V(x,t)\geq g(x)=\varphi(x),
\]
which implies that \eqref{Terminal2} holds true.

If $0<x<x_0$, choose $\tau=T$ and choose $\pi_s$ such that
\[
\frac{\zeta_s}{\zeta_t}[\pi_t'\sigma -\mu'\sigma^{-1} X_t]=(\pi^N_s)':=N\chi_{\big\{0<\frac{\zeta_s}{\zeta_t}X_s<x_0\big\}}I_n', \quad \forall\;N>0,
\]
where $I_n$ is an $n$-dimensional unit column vector. Let $X_s^N=\frac{\zeta_s}{\zeta_t}X_s$. Then using \eqref{dX} results in
\[
dX_s^N=(\pi^N_s)'dW_s,\quad t\leq s\leq T.
\]
It is not hard to obtain
\[
0\leq X_s^N\leq x_0,\quad t\leq s\leq T,
\]
and since
\begin{eqnarray*}
\{0<X_T^N<x_0\}
&=&        \{0<X_s^N=x+NI_n'(W_s-W_t)<x_0,\;t\leq s\leq T\}\\
&\subset&  \{0<x+NI_n'(W_T-W_t)<x_0\},
\end{eqnarray*}
we have
\[
P(0<X_T^N<x_0)\leq P(0<x+NI_n'(W_T-W_t)<x_0)\rightarrow 0,\quad N\rightarrow\infty.
\]
Note that
\[
x_0P(X_T^N=x_0)\leq EX_T^N \leq x_0P(X_T^N=x_0)+x_0P(0<X_T^N<x_0).
\]
Therefore,
\[
\lim\limits_{N\rightarrow\infty}P(X_T^N=x_0)=\frac{EX_T^N}{x_0}=\frac{x}{x_0}, \quad \lim\limits_{N\rightarrow\infty}P(X_T^N=0)=1-\frac{x}{x_0}.
\]
As a result,
\[
\lim\limits_{N\rightarrow\infty}Eg(X_T^N)=\frac{x}{x_0}g(x_0)+\Big(1-\frac{x}{x_0}\Big)g(0) =\frac{x}{x_0}\Big(kx_0+\frac{1}{\gamma}K^\gamma\Big)+\Big(1-\frac{x}{x_0}\Big)\frac{1}{\gamma}K^\gamma =kx+\frac{1}{\gamma}K^\gamma=\varphi(x).
\]
Thus
\[
\sup\limits_{\tau,\;\pi}E_{t,x}\Big(e^{-\beta(\tau-t)}g\Big(\frac{\zeta_\tau}{\zeta_t} X_\tau\Big)\Big)\geq e^{-\beta(T-t)}\lim\limits_{N\rightarrow\infty}Eg(X_T^N)=e^{-\beta(T-t)}\varphi(x).
\]

Meanwhile, similar to \eqref{Terminal1}, we have
\begin{eqnarray*}
\lim\limits_{t\rightarrow T-}\sup\limits_{\tau,\;\pi}E_{t,x}\Big|g(X_\tau)-g\Big(\frac{\zeta_\tau}{\zeta_t} X_\tau\Big)\Big|=0.
\end{eqnarray*}
Therefore,
\begin{eqnarray*}
\liminf\limits_{t\rightarrow T-}V(x,t)
&=&
\liminf\limits_{t\rightarrow T-}\sup\limits_{\tau,\;\pi}E_{t,x}\Big(e^{-\beta(\tau-t)}g(X_\tau)\Big)\\
&\geq&
\liminf\limits_{t\rightarrow T-}\sup\limits_{\tau,\;\pi}E_{t,x}\Big(e^{-\beta(\tau-t)}g\Big(\frac{\zeta_\tau}{\zeta_t} X_\tau\Big)\Big) -\lim\limits_{t\rightarrow T-}\sup\limits_{\tau,\;\pi}E_{t,x}\Big|g(X_\tau)-g\Big(\frac{\zeta_\tau}{\zeta_t} X_\tau\Big)\Big|\\
&\geq&
\varphi(x).
\end{eqnarray*}
\end{proof}
\par
Since the value function is not continuous at the ternimal time $T$, %it does not satisfy its corresponding Hamilton-Jacobi-Bellman (HJB) equation.
we introduce its corresponding HJB equation with the terminal condition $V(x,T-)=\varphi(x),\quad x>0$ in the next subsection.
%for the adjusted value function.

\subsection{The HJB equation}\label{sec:unconstrained}

%Let
%\begin{align}\label{tildeV}
%\widetilde{V}(x,t)=\begin{cases}
%V(x,t),&\quad \text{if }   x\geqslant 0,\quad 0\leqslant t<T;\\
%varphi(x),&\quad\text{if } x\geqslant 0,\quad t=T.
%end{cases}
%\end{align}
%

Applying the dynamic programming principle, we get the following  HJB equation
%\emr{Here, we need to show $\widetilde{V}$ is continuous, otherwise we cannot talk about viscosity solution.}
%
%\begin{eqnarray}                                                \label{V0}
%\min\Big\{-V_t-\max\limits_\pi[\frac{1}{2}(\pi'\sigma\sigma'\pi)V_{xx}+\mu^\prime\pi V_x]-rxV_x+\beta V, V-\frac{1}{\gamma}[(x-b)^++K]^\gamma\Big\}=0
%\end{eqnarray}
%
\begin{eqnarray}\label{V0}
\left\{\begin{array}{l}
\min\Big\{-V_t-\max\limits_\pi[\frac{1}{2}(\pi'\sigma\sigma'\pi)V_{xx}+\mu^\prime\pi V_x]-rxV_x+\beta V, V-\frac{1}{\gamma}[(x-b)^++K]^\gamma\Big\}=0,\\
\hfill \quad x>0, \; 0<t<T, \\ [5mm]
V(0,t)=\frac{1}{\gamma}K^\gamma, \quad 0<t< T, \\ [2mm]
V(x,T-)=\varphi(x),\quad x>0,
\end{array}
\right.
\end{eqnarray}

\emr{Assume} $V_x\geq0$. Note that the Hamiltonian operator
\begin{eqnarray}\label{Hamiltonian}
\max_\pi\Big\{\frac{1}{2}(\pi'\sigma\sigma'\pi)V_{xx}+\mu'\pi V_x\Big\}-rxV_x+rV
\end{eqnarray}
is singularity if $V_{xx}>0$ or $V_{xx}=0,\;V_x>0$.
Thus, $V_{xx}\leq0$. Moreover, if $V_x=0$ holds on $(x_0,t_0)$,
then for any $x\geq x_0$, $V_x(x,t_0)=0$,
which contradicts to $V(x,t)\geq\frac{1}{\gamma}[(x-b)^++K]^\gamma\rightarrow +\infty,\; x\rightarrow+\infty$.
Therefore, if $V(x,t)\in C^{2,1}$ is \emr{an increasing (in $x$)} solutionof \eqref{V0}, it must satisfy
\begin{eqnarray}\label{priori2}
V_{x}>0, V_{xx}<0, \quad x>0, \; 0<t<T.
\end{eqnarray}

Note that the gradient of $\pi'\sigma\sigma'\pi$ with respect to $\pi$ is
\[
\bigtriangledown_\pi(\pi'\sigma\sigma'\pi)=2\sigma\sigma'\pi.
\]
Hence,
\[
\pi^*=-(\sigma\sigma')^{-1}\mu\frac{V_x(x,t)}{V_{xx}(x,t)}.
\]

Applying $V_{xx}<0$, we have
\[
V-\frac{1}{\gamma}[(x-b)^++K]^\gamma\geqslant 0\quad\text{ if and only if }\quad V-\varphi(x)\geqslant 0.
\]
Define $a^2=\mu'(\sigma\sigma')^{-1}\mu$, then the variational inequality  \eqref{V0} is reduce to
\[
\min\Big\{-V_t+\frac{a^2}{2}\frac{V^2_x}{V_{xx}}-rxV_x+\beta V,\; V-\varphi(x)\Big\}=0,
\]

Hence, we formulate our problem into the following variational inequality problem
\begin{eqnarray}\label{V}
\left\{\begin{array}{l}
\min\Big\{-V_t+\displaystyle\frac{a^2}{2}\frac{V^2_x}{V_{xx}}-rxV_x+\beta V, V-\varphi(x)\Big\}=0, \quad x>0, \; 0<t<T, \\ [5mm]
V(0,t)=\frac{1}{\gamma}K^\gamma, \quad 0<t< T, \\ [2mm]
V(x,T-)=\varphi(x),\quad x>0.
\end{array}
\right.
\end{eqnarray}
\emr{We want to show this equation has a (unique) solution $V(x,t)\in C^{2,1}$ which satisfies \eqref{priori2}.
And a verification theorem will ensure this solution is just $\widetilde{V}$ defined in \eqref{tildeV}.
}

%%%%%%%%%%%%%%%%%%%%%%%%%%%%%%%%%%%%%%%%%%

\section{Dual Problem}\label{sec:Dual Problem}\setcounter{equation}{0}
Firstly, assume that
\begin{eqnarray}\label{priori1}
V,\;V_x\in C(Q_x),
\end{eqnarray}
where $Q_x=[0,+\infty)\times(0,T)$ and
\begin{eqnarray}                                  \label{priori3}
\lim\limits_{x\rightarrow+\infty}V(x,t)=+\infty,\quad \lim\limits_{x\rightarrow+\infty}V_x(x,t)=0,\quad \forall
t\in(0,T).
\end{eqnarray}
Later, we will prove the above results in Theorem \ref{thm:V}.

Now define a dual transformation of $V(x,t)$, for any $t\in(0,T)$ (see Pham \cite{Pham}),
\begin{eqnarray}                                  \label{vV}
v(y,t)=\max\limits_{x\geq0}(V(x,t)-xy),\quad y>0,
\end{eqnarray}
then the optimal $x$ to fix $y>0$ satisfies
\begin{eqnarray*}
\left\{\begin{array}{ll}
\partial_x(V(x,t)-xy)=V_x(x,t)-y=0, & \mbox{if } y\leq V_x(0,t), \\ [2mm]
x=0, & \mbox{if } y> V_x(0,t).
\end{array}\right.
\end{eqnarray*}
Define a transformation
\begin{eqnarray}\label{x=I}
x=I(y,t):=\left\{\begin{array}{ll}
(V_x(\cdot,t))^{-1}(y), & \mbox{if } y\leq V_x(0,t), \\ [2mm]
0, & \mbox{if } y> V_x(0,t).
\end{array}\right.
\end{eqnarray}
Owing to \eqref{priori2} and \eqref{priori3}, $I(y,t)$ is continuously decreasing in $y$ and
$\lim\limits_{y\rightarrow0+}I(y,t)=+\infty$. Thus
\begin{eqnarray}                                                             \label{vV}
v(y,t)=V(I(y,t),t)-I(y,t)y,
\end{eqnarray}

Define the dual transformation of $\varphi(x)$ as
\[
\psi(y)=\max\limits_{x\geq0}(\varphi(x)-xy),\quad y>0,
\]
then the optimal $x$ to fix $y$, which we denote by $x_\varphi(y)$, is
\[
x_\varphi(y)=
\left\{\begin{array}{ll}
y^{\frac{1}{\gamma-1}}-(K-b), & \mbox{for } 0<y<k, \\ [2mm]
\in [0,\widehat{x}], & \mbox{for } y=k,\\ [2mm]
0, & \mbox{for } y>k,
\end{array}\right.
\]
and
\begin{eqnarray}
\psi(y) & \!\!\!=\!\!\! & \varphi(x_\varphi(y))-x_\varphi(y)y \nonumber \\ [2mm]
\label{psi} & \!\!\!=\!\!\! & \left\{\begin{array}{ll}
\frac{1-\gamma}{\gamma}y^{\frac{\gamma}{\gamma-1}}+(K-b)y, & \mbox{for } 0<y<k, \\ [2mm]
\frac{1}{\gamma}K^{\gamma}, & \mbox{for } y\geq k,
\end{array}\right.
\end{eqnarray}
(see Fig 3.1).

\bigskip

\begin{center}
\begin{picture}(250,100)
    \put(50,10){\vector(1,0){150}}
    \put(60,0){\vector(0,1){95}}
    \put(80,80){$\psi(y)$}
    \put(115,-5){$k$}\put(120,10){\circle*{2}}
    \put(205,-2){$y$}
    \qbezier(70,95)(85,40)(120,30)
    \put(60,30){\line(1,0){10}}\put(80,30){\line(1,0){10}}\put(100,30){\line(1,0){10}}\put(120,30){\line(1,0){50}}
    \put(30,25){$\frac{1}{\gamma}K^\gamma$}
    \put(80,-30){Fig 3.1 \;$\psi(y)$.}
    \end{picture}
\end{center}
\bigskip
\bigskip
\bigskip

It follows from \eqref{psi} and \eqref{k x} that we get
\begin{eqnarray*}
\psi'(y)=
\left\{\begin{array}{ll}
-y^{\frac{1}{\gamma-1}}+(K-b)<-\widehat{x}<0, & \mbox{for } 0<y<k, \\ [2mm]
0, & \mbox{for } y> k.
\end{array}\right.
\end{eqnarray*}
and
\begin{eqnarray*}
\psi''(y)=
\left\{\begin{array}{ll}
\frac{1}{1-\gamma}y^{\frac{1}{\gamma-1}-1}>0, & \mbox{for } 0<y<k, \\ [2mm]
0, & \mbox{for } y> k.
\end{array}\right.
\end{eqnarray*}
It is obvious that
\[
\varphi(x)=\min\limits_{y>0}(\psi(y)+xy).
\]

It follows from \eqref{vV} that we have
\begin{eqnarray}
\label{vy=-I}
&&v_y(y,t)=V_x(I(y,t),t)I_y(y,t)-y I_y(y,t)-I(y,t)=-I(y,t), \\ [2mm]
\label{Vxx}
&&v_{yy}(y,t)=-I_y(y,t)=\frac{-1}{V_{xx}(I(y,t),t)}\chi_{\{y\leq V_x(0,t)\}}, \\ [2mm]
&&v_t(y,t)=V_t(I(y,t),t)+V_x(I(y,t),t)I_t(y,t)-y I_t(y,t)=V_t(I(y,t),t). \nonumber
\end{eqnarray}
Thus, for any $y>0$, set $x=I(y,t)$. If $y\leq V_x(0,t)$, then
\begin{eqnarray}                                                     \label{Lv=LV}
-v_t-\frac{a^2}{2}y^2v_{yy}-(\beta-r)yv_y+\beta v=-V_t+\frac{a^2}{2}\frac{V^2_x}{V_{xx}}-rxV_x+\beta V\geq0;
\end{eqnarray}
If $y> V_x(0,t)$, then $x=0$, $v(y,t)=V(0,t)=\frac{1}{\gamma}K^\gamma$, $v_t(y,t)=V_t(0,t)=0$,
$v_y(y,t)=v_{yy}(y,t)=0$, which implies
\begin{eqnarray}\label{Lv>0}
v(y,t)=\psi(y),\;-v_t-\frac{a^2}{2}y^2v_{yy}-(\beta-r)yv_y+\beta v=\frac{\beta}{\gamma}K^\gamma>0 .
\end{eqnarray}
In the above case of $y> V_x(0,t)$, since $V_x(0,t)\geq k$ (which we will prove in Theorem \ref{thm:case}),
we have $y>k$ as well as $\psi(y)=\frac{1}{\gamma}K^\gamma$. Combining \eqref{Lv=LV} and \eqref{Lv>0} yields
\begin{eqnarray}                                                      \label{condition1}
-v_t-\frac{a^2}{2}y^2v_{yy}-(\beta-r)yv_y+\beta v\geq0,\quad y>0,\;0<t<T.
\end{eqnarray}

By the definition of $v$ and $\psi$, we have
\begin{eqnarray}
&&V(x,t)\geq\varphi(x),\;\forall\;x\geq0\nonumber\\
&\Rightarrow&\max\limits_{x\geq0}(V(x,t)-xy)\geq\max\limits_{x\geq0}(\varphi(x)-xy), \;\forall\;y>0\nonumber\\\label{condition2}
&\Rightarrow& v(y,t)\geq \psi(y), \;\forall\;y>0.
\end{eqnarray}
On the other hand,
\begin{eqnarray*}
\begin{array}{ll}
& v(y,t)>\psi(y) \\ [2mm]
\Rightarrow & V(I(y,t),t)-I(y,t)y>\max\limits_{x\geq0}(\varphi(x)-xy) \\ [2mm]
\Rightarrow & V(I(y,t),t)-I(y,t)y>\varphi(I(y,t))-I(y,t)y \\ [2mm]
\Rightarrow & V(I(y,t),t)>\varphi(I(y,t)), \\ [2mm]
\end{array}
\end{eqnarray*}
and by the variational inequality in \eqref{V},
\begin{eqnarray*}
\begin{array}{ll}
& V(I(y,t),t)>\varphi(I(y,t)) \\  [2mm]
\Rightarrow & \Big(-V_t+\displaystyle\frac{a^2}{2}\frac{V^2_x}{V_{xx}}-rxV_x+\beta V\Big)(I(y,t),t)=0.
\end{array}
\end{eqnarray*}
together with \eqref{Lv=LV} yields
\begin{eqnarray}\label{condition3}
v(y,t)>\psi(y)\Rightarrow & \Big(-v_t-\displaystyle\frac{a^2}{2}y^2v_{yy}-(\beta-r)yv_y+\beta v\Big)(y,t)=0.
\end{eqnarray}

Combining the above equation with \eqref{condition1}, \eqref{condition2} and \eqref{condition3}, we obtain
\begin{eqnarray}\label{v}
\left\{
\begin{array}{l}
\min\{-v_t-\displaystyle\frac{a^2}{2}y^2v_{yy}-(\beta-r)yv_y+\beta v, v-\psi\}=0,\quad y>0,\;0<t<T, \\ [5mm]
v(y,T-)=\psi(y),\quad y>0.
\end{array}
\right.
\end{eqnarray}

\begin{remark}\sl
The equation in \eqref{v} is degenerate on the boundary $y=0$. According to Fichera's theorem (see Ole\u{i}nik
and Radkevi\'{c} \cite{OO}), we must not put the boundary condition on $y=0$.
\end{remark}
%%%%%%%%%%%%%%%%%%%%%%%%%%%%%%%%%%%%%%%%%%

\section{The solution and the free boundary of \eqref{v}}\label{sec:proof}
\setcounter{equation}{0}
Now we find the solution of \eqref{v} which is under liner growth condition.
\begin{theorem}
\sl                                              \label{thm:existence of u}
Problem \eqref{v} has unique solution $v(y,t)\in W^{2,1}_{p,\;loc}(Q_y\setminus B_\varepsilon(k,T))\cap C(\overline{Q_y})$ for any $p>2$ and small $\varepsilon>0$.
Moreover
\begin{eqnarray}                                                \label{v1}
&&\psi(y)\leq v \leq A(e^{B(T-t)}y^{\frac{\gamma}{\gamma-1}}+1),\\ \label{v3}
&&v_t\leq0,\\                                                   \label{v4}
&&v_y \leq0,\\                                                     \label{v5}
&&v_{yy}\geq0.
\end{eqnarray}
where $Q_y=(0,+\infty)\times(0,T)$, $B_{\varepsilon}(k,T)$ is the disk with center (k,T) and radius $\varepsilon$, $A=\max\{\frac{1-\gamma}{\gamma},\frac{1}{\gamma}K^\gamma,|K-b|k\}$, $B=\frac{a^2}{2}\frac{\gamma}{(\gamma-1)^2} +\frac{\beta-r\gamma}{\gamma-1}$.
\end{theorem}
\begin{proof}
According to the results of existence and uniqueness of $W^{2,1}_{p}$ solutions\cite{Lady},
the solution of system \eqref{v} can be proved by a standard penalty method, furthermore, by Sobolev embedding theorem,
\begin{eqnarray}\label{continue}
v\in C((Q_y\cup\{t=0,T\})),\quad v_y\in C((Q_y\cup\{t=0,T\})\setminus(k,T))
\end{eqnarray}
(see Friedman \cite{Fr}).
Here, we omit the details. The first inequality in \eqref{v1} follows from \eqref{v} directly, we now prove the second inequality in \eqref{v1}.

Denote
\begin{eqnarray*}
w(y,t)=A(e^{B(T-t)}y^{\frac{\gamma}{\gamma-1}}+1).
\end{eqnarray*}
Then
\begin{eqnarray*}
&&-w_t-\frac{a^2}{2}y^2w_{yy}-(\beta-r)yw_y+\beta w \\
&=&Ae^{B(T-t)}y^{\frac{\gamma}{\gamma-1}}\Big(B-\frac{a^2}{2}\Big(\frac{\gamma}{\gamma-1}\Big)\Big(\frac{\gamma}{\gamma-1}-1\Big) -(\beta-r)\Big(\frac{\gamma}{\gamma-1}\Big)+\beta\Big)+\beta A\\
&=&Ae^{B(T-t)}y^{\frac{\gamma}{\gamma-1}}\Big(B-\frac{a^2}{2}\frac{\gamma}{(\gamma-1)^2} -\frac{\beta-r\gamma}{\gamma-1}\Big)+\beta A\\
&\geq&0,
\end{eqnarray*}
and $w(y,t)\geq w(y,T)\geq\psi(y)$.
Using the comparison principle of variational inequality (see Friedman \cite{Fr}), we know that $w$ is a super solution of \eqref{v}.

Next we prove \eqref{v3}. Let $\widetilde{v}(y,t)=v(y,t-\delta)$ for small $\delta>0$, then $\widetilde{v}$ satisfies
\begin{eqnarray*}
\left\{\begin{array}{l}
\min\{-\widetilde{v}_t-\frac{a^2}{2}y^2\widetilde{v}_{yy}-(\beta-r)y\widetilde{v}_y+\beta\widetilde{v},
\widetilde{v}-\psi(y)\}=0,\quad y>0,\;\delta<t<T, \\ [2mm]
\widetilde{v}(y,T)\geq\psi(y),\quad y>0.
\end{array}\right.
\end{eqnarray*}
Hence, by the comparison principle, we have $\widetilde{v}\geq v$, i.e. $v_t\leq0$.

Define
\begin{eqnarray*}
{\cal \varepsilon R}_y &=& \{(y,t)\in Q_y|v=\psi\},\quad\hbox{exercise region}, \\
\quad{\cal CR}_y &=& \{(y,t)\in Q_y|v>\psi\},\quad\hbox{continuation region}.
\end{eqnarray*}

Note that $k$ is the only discontinuity point of $\psi'(y)$ and $\psi''(y)$.
Now, we claim $(k,t)$ could not be contained in ${\cal \varepsilon R}_y$ for all $t\in(0,T)$.
Otherwise, if $(k,t_0)\in{\cal \varepsilon R}_y$ for some $t_0<T$, then it belongs
to the minimum points of $v-\psi(y)$, thus $v_y(k-,t_0)\leq\psi^{\prime}(k-)<\psi^{\prime}(k+)\leq v_y(k+,t_0)$, which implies $v_y$ does not continue
at the point $(k,t_0)$ so as to yields a contradiction to \eqref{continue}.

Here, we present the proof of \eqref{v4} and \eqref{v5}. Recalling that
\begin{eqnarray*}
&&\psi'(y)<0,\quad\psi''(y)>0,\quad y\in(0,k), \\ [2mm]
&&\psi'(y)=0,\quad\psi''(y)=0,\quad y\in(k,+\infty).
\end{eqnarray*}
Thanks to $v$ gets the minimum value in ${\cal \varepsilon R}_y$,
$v_y=\psi'\leq0$ in ${\cal \varepsilon R}_y$.
Moreover, $v_y(y,T)=\psi'\leq0$.
Taking the derivative for the following equation
\[
-v_t-\frac{a^2}{2}y^2v_{yy}-(\beta-r)yv_y+\beta v=0 \quad \mbox{in } \quad {\cal CR}_y
\]
with respect to $y$ leads to
\begin{eqnarray}\label{vy}
-\partial_tv_y-\frac{a^2}{2}y^2\partial_{yy}v_y-(a^2+\beta-r)y\partial_yv_y+rv_y=0 \quad \mbox{in } \quad {\cal CR}_y.
\end{eqnarray}
Note that $v_y=\psi'\leq 0$ on $\partial({\cal CR}_y)$, where $\partial({\cal CR}_y)$ is the boundary
of ${\cal CR}_y$ in the interior of $Q_y$, using the maximum principle we obtain \eqref{v4}.

In addition, $v\geq\psi$, together with $v=\psi,\;v_y=\psi'$ in ${\cal \varepsilon R}_y$ yields $v_{yy}= \psi''\geq0$ in ${\cal \varepsilon R}_y$.
It is not hard to prove that
\[
\lim\limits_{{\cal CR}_y\ni y\rightarrow \partial({\cal CR}_y)}v_{yy}(y,t)\geq0.
\]
and $v_{yy}(y,T)=\psi''\geq0$.
Taking the derivative for equation \eqref{vy} with respect to $y$, we obtain
\[
-\partial_tv_{yy}-\frac{a^2}{2}y^2\partial_{yy}v_{yy}-2a^2y\partial_yv_{yy}+(r-a^2)v_{yy}=0 \quad \mbox{in } \quad
{\cal CR}_y.
\]
Using the maximum principle, we obtain
\begin{eqnarray}\label{vyy2}
v_{yy}\geq0 \quad\hbox{in}\quad {\cal CR}_y.
\end{eqnarray}
\end{proof}

Define free boundaries
\begin{eqnarray}\label{h(t)}
&&h(t)=\inf\{y\in[0,k]|v(y,t)=\psi(y)\},\quad 0<t<T,\\\label{g(t)}
&&g(t)=\sup\{y\in[0,k]|v(y,t)=\psi(y)\},\quad 0<t<T,\\\label{f(t)}
&&f(t)=\inf\{y\in[k,+\infty)|v(y,t)=\psi(y)\},\quad 0<t<T,
\end{eqnarray}
Owing to $\partial_t(v(y,t)-\psi(y))=v_t\leq0$, functions $h(t)$ and $f(t)$ are decreasing in $t$ and $g(t)$ is increasing in $t$.

Substituting the first expression of \eqref{psi} into the equation in \eqref{v} yields
\begin{eqnarray}
&&-\partial_t\psi-\frac{a^2}{2}y^2\partial_{yy}\psi-(\beta-r)y\partial_y\psi+\beta \psi \nonumber\\
&=&\frac{a^2}{2}\Big(\frac{\gamma}{\gamma-1}-1\Big)y^\frac{\gamma}{\gamma-1}
-(\beta-r)y\Big[-y^{\frac{1}{\gamma-1}}+(K-b)\Big]+\beta\Big[\frac{1-\gamma}{\gamma}y^\frac{\gamma}{\gamma-1}+(K-b)y\Big] \nonumber\\                    \label{Lpsi}
&=&\Big(\frac{\beta-r\gamma}{\gamma}-\frac{a^2}{2}\frac{1}{1-\gamma}\Big)y^\frac{\gamma}{\gamma-1}+r(K-b)y,\quad y<k,
\end{eqnarray}
and note that
\[
-\partial_t\psi-\frac{a^2}{2}y^2\partial_{yy}\psi-(\beta-r)y\partial_y\psi+\beta \psi=\frac{\beta}{\gamma}K^\gamma>0,\quad y>k.
\]
Denote the right hand side of \eqref{Lpsi} by $\Psi(y)$. It is not hard to see that
\begin{eqnarray}                                                            \label{CRysubset}
{\cal \varepsilon R}_y\subset [\{\Psi(y)\geq0,\;y<k\}\cup (k,+\infty)]\times(0,T).
\end{eqnarray}

\begin{lemma}\label{lem:inCR}
\sl
The set ${\cal \varepsilon R}_y$ is expressed as
\begin{eqnarray}                                                                        \label{CRy}
{\cal \varepsilon R}_y=\{(y,t)\in Q_y|h(t)\leq y\leq g(t)\}\cup\{(y,t)\in Q_y|y\geq f(t)\}.
\end{eqnarray}
\end{lemma}
\begin{proof}
By the definitions of $h(t),\;g(t)$ and $f(t)$, we get
\[
{\cal \varepsilon R}_y\subset\{(y,t)\in Q_y|h(t)\leq y\leq g(t)\}\cup\{(y,t)\in Q_y|y\geq f(t)\}.
\]
Now, we prove
\begin{eqnarray}\label{subset}
\Omega:=\{(y,t)\in Q_y|h(t)\leq y\leq g(t)\}\subset{\cal \varepsilon R}_y.
\end{eqnarray}
Since $\{(h(t),t),\;(g(t),t)\}\cap Q_y\subset {\cal \varepsilon R}_y\cap\{y<k\}\subset \{\Psi\geq0\}$ and $\{\Psi\geq0\}$ is a connected region, we have
\[
\Omega\subset \{\Psi\geq0\}.
\]
Assume that \eqref{subset} is false.
Since ${\cal CR}_y$ is an open set,
there exists an open subset  $\cal N$ such that ${\cal N}\subset \Omega$ and $\partial_p{\cal N}\subset {\cal \varepsilon R}_y$,
where $\partial_p{\cal N}$ is the parabolic boundary of ${\cal N}$.
Thus,
\begin{eqnarray}\label{vpsi}
\left\{\begin{array}{l}
-v_t-\displaystyle\frac{a^2}{2}y^2v_{yy}-(\beta-r)yv_y+\beta v=0 \quad \hbox{in}\;{\cal N}, \\ [2mm]
-\psi_t-\displaystyle\frac{a^2}{2}y^2\psi_{yy}-(\beta-r)y\psi_y+\beta \psi\geq0\quad \hbox{in}\;{\cal N}, \\ [2mm]
v=\psi\quad \hbox{on}\quad\partial_p{\cal N}.
\end{array}\right.
\end{eqnarray}
By the comparison principle, $v\leq\psi$ in ${\cal N}$, which implies ${\cal N}=\emptyset$.

Similar proof yields
\begin{eqnarray*}
\{(y,t)\in Q_y|y\geq f(t)\}\subset{\cal \varepsilon R}_y.
\end{eqnarray*}
Therefore, the desired result \eqref{CRy} holds.
\end{proof}

Thanks to Lemma \ref{lem:inCR}, $h(t)$, $g(t)$ and $f(t)$ are three free boundaries of \eqref{v}.

\begin{theorem}\label{thm:classify}
\sl
The free boundaries
$h(t),\;g(t)$ and $f(t)\in C^\infty(0,T)$ and have the following classification

\noindent {\bf Case I: $\beta\geq\frac{a^2}{2}\frac{\gamma}{1-\gamma}+r\gamma$, $\Psi (k)\geq 0$.}
\[
h(t)\equiv0\leq g(t)\leq g(T-)=k=f(T-)\leq f(t),
\]
see Fig 4.1.

\noindent {\bf Case II: $\beta\geq\frac{a^2}{2}\frac{\gamma}{1-\gamma}+r\gamma$, $\Psi (k)< 0$.}

If $\beta>\frac{a^2}{2}\frac{\gamma}{1-\gamma}+r\gamma$,
\[
h(t)\equiv0\leq g(t)\leq g(T-)=
\Big(\frac{-r(K-b)}{\frac{\beta-r\gamma}{\gamma}-\frac{a^2}{2}\frac{1}{1-\gamma}}\Big)^{\gamma-1}<k=f(T-)\leq f(t),
\]
see Fig 4.2.

If $\beta=\frac{a^2}{2}\frac{\gamma}{1-\gamma}+r\gamma$,
\[
{\cal \varepsilon R}_y\cap\Big((0,k)\times(0,T)\Big)=\emptyset,\quad k=f(T-)\leq f(t),
\]
see Fig 4.4.

\noindent {\bf Case III: $\beta<\frac{a^2}{2}\frac{\gamma}{1-\gamma}+r\gamma$, $\Psi (k)> 0$.}
\[
\Big(\frac{-r(K-b)}{\frac{\beta-r\gamma}{\gamma}-\frac{a^2}{2}\frac{1}{1-\gamma}}\Big)^{\gamma-1}
=h(T-)\leq h(t)\leq g(t)\leq g(T-)=k=f(T-)\leq f(t),
\]
see Fig 4.3.

\noindent {\bf Case IV: $\beta<\frac{a^2}{2}\frac{\gamma}{1-\gamma}+r\gamma$, $\Psi (k)\leq 0$.}
\[
{\cal \varepsilon R}_y\cap(0,k)\times(0,T)=\emptyset,\quad k=f(T-)\leq f(t),
\]
see Fig 4.4.
%
%where, $\Psi(y)$ is defined by \eqref{Lpsi}.
\end{theorem}

\begin{center}
  \begin{picture}(0,0)
    \put(-180,-90){\vector(1,0){150}}
    \put(-170,-100){\vector(0,1){95}}
    \put(-190,-20){$T$}
    \put(-170,-20){\line(1,0){140}}
    \put(-90,-15){$k$}
    \put(-800,-20){\circle*{2}}
    \qbezier(-130,-90)(-110,-35)(-80,-20)
    \put(-160,-70){${\cal \varepsilon R}_y$}
    \put(-90,-70){${\cal CR}_y$}
    \put(-40,-70){${\cal \varepsilon R}_y$}
    \qbezier(-40,-90)(-50,-30)(-80,-20)
    \put(-120,-40){$g(t)$}
    \put(-60,-40){$f(t)$}
    \put(-25,-102){$y$}
    \put(-180,-120){Fig 4.1. $\beta\geq\frac{a^2}{2}\frac{\gamma}{1-\gamma}+r\gamma$, $\Psi (k)\geq 0$. }                                                                \end{picture}
  \begin{picture}(0,0)
    \put(50,-90){\vector(1,0){150}}
    \put(60,-100){\vector(0,1){95}}
    \put(40,-20){$T$}
    \put(60,-20){\line(1,0){140}}
    \put(145,-15){$k$}
    \put(150,-20){\circle*{2}}
    \put(70,-70){${\cal \varepsilon R}_y$}
    \put(130,-70){${\cal CR}_y$}
    \put(185,-70){${\cal \varepsilon R}_y$}
    \qbezier(190,-90)(180,-30)(150,-20)
    \qbezier(90,-90)(100,-30)(130,-20)
    \put(120,-15){$y_T$}
    \put(205,-102){$y$}
    \put(90,-40){$g(t)$}
    \put(170,-40){$f(t)$}
    \put(50,-120){Fig 4.2. \;$\beta>\frac{a^2}{2}\frac{\gamma}{1-\gamma}+r\gamma$, $\Psi (k)< 0$.}
  \end{picture}
\end{center}
\vspace{4cm}

\begin{center}
  \begin{picture}(0,0)
    \put(50,-90){\vector(1,0){150}}
    \put(60,-100){\vector(0,1){95}}
    \put(40,-20){$T$}
    \put(60,-20){\line(1,0){140}}
    \put(140,-15){$k$}
    \put(130,-70){${\cal CR}_y$}
    \put(190,-70){${\cal \varepsilon R}_y$}
    \qbezier(190,-90)(180,-30)(150,-20)
    \put(-160,-15){$y_T$}
    \put(205,-102){$y$}
    \put(-120,-40){$g(t)$}
    \put(-55,-40){$f(t)$}
    \put(-150,-40){$h(t)$}
    \put(-180,-120){Fig 4.3. \;$\beta<\frac{a^2}{2}\frac{\gamma}{1-\gamma}+r\gamma$, $\Psi (k)> 0$.}
    \end{picture}
    \begin{picture}(0,-0)
    \put(-180,-90){\vector(1,0){150}}
    \put(-170,-100){\vector(0,1){95}}
    \put(-190,-20){$T$}
    \put(-170,-20){\line(1,0){140}}
    \put(-95,-15){$k$}
    \put(-80,-20){\circle*{2}}
    \qbezier(-130,-90)(-110,-30)(-80,-20)
    \put(-140,-60){${\cal \varepsilon R}_y$}
    \put(-100,-70){${\cal CR}_y$}
    \put(-45,-60){${\cal \varepsilon R}_y$}
    \put(-170,-70){${\cal CR}_y$}
    \qbezier(-40,-90)(-50,-30)(-80,-20)
    \qbezier(-140,-90)(-150,-30)(-160,-20)
    \put(-25,-102){$y$}
    \put(175,-40){$f(t)$}
    \put(40,-120){Fig 4.4. \;$\beta<\frac{a^2}{2}\frac{\gamma}{1-\gamma}+r\gamma$, $\Psi (k)\leq 0$,}
    \put(40,-140){\quad\quad\quad or $\beta=\frac{a^2}{2}\frac{\gamma}{1-\gamma}+r\gamma$, $\Psi (k)< 0$.}
    .
  \end{picture}
\end{center}
\vspace{5cm}

\begin{proof}
By the method of \cite{Fr}, we could prove $h(t),\;g(t),\;f(t)\in C^\infty(0,T)$, we omit the details.

Here, we only prove the  results in {\bf Case II}, the remaining situations are similar.
If $\beta>\frac{a^2}{2}\frac{\gamma}{1-\gamma}+r\gamma$ and $\Psi(k)<0$, then $K<b$, Denote $y_T=\Big(\frac{-r(K-b)}{\frac{\beta-r\gamma}{\gamma}-\frac{a^2}{2}\frac{1}{1-\gamma}}\Big)^{\gamma-1}$,
then $\Psi(k)<0$ implies $y_T < k$.
By \eqref{CRysubset} and $\{\Psi\geq0\}=(0,y_T]$,
\[
{\cal \varepsilon R}_y\subset\Big((0,y_T]\cup(k,\infty)\Big)\times (0,T).
\]
thus
\[
0\leq h(t)\leq g(t)\leq y_T<k \leq f(t).
\]

Now, we prove $h(t)\equiv0$.
Set ${\cal N}:=\{(y,t)|0<y\leq h(t),\;0<t<T\}$. It follows from \eqref{vpsi} that we have $v\leq\psi$ in ${\cal N}$.
By the definition of $h(t)$, ${\cal N}=\emptyset$ as well as $h(t)\equiv0$.

Here, we aim to prove $f(T-):=\lim\limits_{t\uparrow T}f(t)=k$.
Otherwise, if $f(T-)>k$, then there exists a contradiction that
\begin{eqnarray*}
0
&=&-v_t-\frac{a^2}{2}y^2v_{yy}-(\beta-r)yv_y+\beta v\\
&=&-v_t-\frac{a^2}{2}y^2\psi_{yy}-(\beta-r)y\psi_y+\beta \psi =-\partial_tv+\frac{1}{\gamma}K^\gamma> 0,\quad k<y<f(T-),\;t=T.
\end{eqnarray*}
So $f(T-)=k$.
The proof of $g(T-)=y_T$ is similar that if $g(T-)>y_T$, there exists contradiction
\begin{eqnarray*}
0
&=&-v_t-\frac{a^2}{2}y^2v_{yy}-(\beta-r)yv_y+\beta v\\
&=&-v_t-\frac{a^2}{2}y^2\psi_{yy}-(\beta-r)y\psi_y+\beta \psi =-\partial_tv+\Psi(y)> 0,\quad g(T-)<y<y_T,\;t=T.
\end{eqnarray*}
If $\beta=\frac{a^2}{2}\frac{\gamma}{1-\gamma}+r\gamma$ and $\Psi(k)<0$, then $K<b$ as well as $\Psi(y)<0$ for all $0<y<k$, thus $(0,k]\times(0,T)\subset {\cal CR}_y$, so $h(t),\; g(t)$ do not exist.
\end{proof}

%%%%%%%%%%%%%%%%%%%%%%%%%%%%%%%%%%%%%%%%%%

\section{The solution and the free boundary of original problem \eqref{V}}\label{sec:original}
\setcounter{equation}{0}
\begin{lemma}                                                \label{lem:vyy}
\begin{eqnarray}                                             \label{vyy}
v_{yy}>0,\quad0<y<f(t),\quad0<t<T.
\end{eqnarray}
\end{lemma}
\begin{proof}
 Apply strong maximum principle,
\[
v_{yy}>0\quad in\quad {\cal CR}_y,
\]
together with
\[
v_{yy}=\psi''>0\quad in\quad {\cal \varepsilon R}_y\cap\Big((0,k)\times(0,T)\Big),
\]
then \eqref{vyy} is true.
\end{proof}

\begin{lemma}                                                \label{lem:vy0}
\begin{eqnarray}                                             \label{vy0}
\lim\limits_{y\rightarrow0+}v_y(y,t)=-\infty,\quad 0<t<T.\\  \label{vyft}
\lim\limits_{y\rightarrow f(t)-}v_y(y,t)=0,\quad 0<t<T.
\end{eqnarray}
\end{lemma}
\begin{proof}
For any $t\in(0,T)$, it is not hard to see that $\lim\limits_{y\rightarrow 0+}v(y,t)\geq\lim\limits_{y\rightarrow 0+}\psi(y)=+\infty$, and by $v_{yy}\geq0$, thus for some fix $y_0>0$,
\[
v_y(y,t)\leq \frac{v(y_0,t)-v(y,t)}{y_0-y}\rightarrow-\infty,\quad y\rightarrow0+.                                                                                  \]
\eqref{vyft} is due to $v_y$ is continuously through the free boundary $y=f(t)$.
\end{proof}

Thanks to Lemma \ref{lem:vyy} and Lemma \ref{lem:vy0}, we can define a transformation
\begin{eqnarray*}
y=J(x,t)=
\left\{
  \begin{array}{ll}
    (v_y(\cdot,t))^{-1}(-x), & \hbox{for}\quad x>0;\\
    f(t), & \hbox{for}\quad x=0,\\
  \end{array}
\right.,
\;0<t<T,
\end{eqnarray*}
then $J(x,t)\in C([0,+\infty)\times(0,T))$ and is decreasing to $x$.

\begin{lemma}                                                \label{lem:JxT}
\begin{eqnarray}                                             \label{Jx0}
&&\lim\limits_{x\rightarrow 0+}J(x,t)=f(t),\quad 0<t<T,\\      \label{Jxinfty}
&&\quad\lim\limits_{x\rightarrow +\infty}J(x,t)=0,\quad 0<t<T,\\ \label{JxT}
&&\lim\limits_{t\rightarrow T-}J(x,t)=\varphi'(x),\quad x\geq0.
\end{eqnarray}
\end{lemma}
\begin{proof}
\eqref{Jx0} and \eqref{Jxinfty} are the results of Lemma \ref{lem:vy0}. Now we prove \eqref{JxT}.

{\bf The case of $x>\widehat{x}$.}
Owing to the regularity of $v_y$ on $t=T$,
\[
\lim\limits_{t\rightarrow T-}v_y(y,t)=\psi'(y),\quad 0<y<k.
\]
Notice that $\psi'(y)$ maps onto $(-\infty,-\widehat{x})$ for $0<y<k$, and $\psi''(y)>0,\;0<y<k$, thus
\begin{eqnarray}                                                      \label{limJ1}
\lim\limits_{t\rightarrow T-}J(x,t)=\lim\limits_{t\rightarrow T-}(v_y(\cdot,t))^{-1}(-x)=(\psi'(\cdot))^{-1}(-x)=\varphi'(x),\quad x>\widehat{x}.
\end{eqnarray}

{\bf The case of $0\leq x\leq\widehat{x}$}.
Due to $(v_y(\cdot,t))^{-1}(-x)$ is decreasing to $x$ for all $t\in(0,T)$,
\begin{eqnarray}                                                      \label{vvv}
(v_y(\cdot,t))^{-1}(-x)\leq(v_y(\cdot,t))^{-1}(0)=f(t).
\end{eqnarray}
If $y<k$, then $\lim\limits_{t\rightarrow T-}v_y(y,t)=\psi'(y)<-\widehat{x}\leq-x$, so when $t$ is sufficient close to $T$,
\[
v_y(y,t)<-x,
\]
thus
\[
(v_y(\cdot,t))^{-1}(-x)>y,
\]
hence
\[
\liminf\limits_{t\rightarrow T-}(v_y(\cdot,t))^{-1}(-x)\geq y,
\]
by the arbitrariness of $y<k$, we see that
\[
\liminf\limits_{t\rightarrow T-}(v_y(\cdot,t))^{-1}(-x)\geq k.
\]
Together with \eqref{vvv},
\[
k\leq\liminf\limits_{t\rightarrow T-}(v_y(\cdot,t))^{-1}(-x)\leq\limsup\limits_{t\rightarrow T-}(v_y(\cdot,t))^{-1}(-x)\leq\lim\limits_{t\rightarrow T-}f(t)=k,
\]
thus
\begin{eqnarray}                                                      \label{limJ2}
\lim\limits_{t\rightarrow T-}J(x,t)=\lim\limits_{t\rightarrow T-}(v_y(\cdot,t))^{-1}(-x)=k=\varphi'(x),\quad 0\leq x\leq \widehat{x}.
\end{eqnarray}
\eqref{JxT} follows from \eqref{limJ1} and \eqref{limJ2}.
\end{proof}

Now, we set
\begin{eqnarray}                                  \label{Vmin}
\widehat{V}(x,t)=\min\limits_{y>0}(v(y,t)+xy).
\end{eqnarray}
\begin{theorem}                                                \label{thm:V}
$\widehat{V}$ is the strong solution of \eqref{V} and satisfies the following
\begin{eqnarray}                                  \label{V1}
\widehat{V},\;\widehat{V}_x\in C(Q_x).
\end{eqnarray}
Moreover,
\begin{eqnarray}                                  \label{V2}
\widehat{V}_t\leq0,\quad \widehat{V}_x>0,\quad \widehat{V}_{xx}<0,\quad (x,t)\in Q_x.
\end{eqnarray}
\begin{eqnarray}                                  \label{V3}
\lim\limits_{x\rightarrow+\infty}\widehat{V}(x,t)=+\infty,\quad \lim\limits_{x\rightarrow+\infty}\widehat{V}_x(x,t)=0,\quad \forall
t\in(0,T).
\end{eqnarray}
\end{theorem}
\begin{proof}
From Lemma \ref{lem:vyy} and Lemma \ref{lem:vy0}, it is easily seen that $J(x,t)\in \arg \min\limits_{y>0}(v(y,t)+xy)$ for all $(x,t)\in Q_x$, thus
\[
\widehat{V}(x,t)=v(J(x,t),t)+xJ(x,t),\quad(x,t)\in Q_x.
\]
In addition
\begin{eqnarray}                             \label{Vx}
&&\widehat{V}_x(x,t)=v_y(J(x,t),t)J_x(x,t)+xJ_x(x,t)+J(x,t)=J(x,t)\geq0,\\\label{Vxx}
&&\widehat{V}_{xx}(x,t)=J_x(x,t)=\partial_x[(v_y(\cdot,t))^{-1}(x)]=\frac{-1}{v_{yy}(J(x,t),t)}<0,\\\label{Vt}
&&\widehat{V}_t(x,t)=v_y(J(x,t),t)J_t(x,t)+v_t(J(x,t),t)+xJ_t(x,t)=v_t(J(x,t),t)\leq0,
\end{eqnarray}
Due to $J(x,t)\in C(Q_x)$, \eqref{V1} is true.
Moreover
\begin{eqnarray*}
\lim\limits_{x\rightarrow+\infty}\widehat{V}(x,t)
&=&\lim\limits_{x\rightarrow+\infty}[v(J(x,t),t)+xJ(x,t)]\\ &\geq&\lim\limits_{x\rightarrow+\infty}v(J(x,t),t)=\lim\limits_{y\rightarrow0+}v(y,t)=+\infty,\quad 0<t<T.
\end{eqnarray*}
\[
\lim\limits_{x\rightarrow+\infty}\widehat{V}_x(x,t)=\lim\limits_{x\rightarrow+\infty}J(x,t)=0,\quad 0<t<T.
\]
Here, we verify $\widehat{V}$ is the strong solution of \eqref{V}.
Firstly,
\[
\lim\limits_{x\rightarrow 0+}\widehat{V}(x,t)=\lim\limits_{x\rightarrow 0+}[v(J(x,t),t)+xJ(x,t)]=v(f(t),t)=\frac{1}{\gamma}K^\gamma,\quad t\in (0,T),
\]
\begin{eqnarray*}
\lim\limits_{t\rightarrow T-}\widehat{V}(x,t)
&=&\lim\limits_{t\rightarrow T-}[v(J(x,t),t)+xJ(x,t)]\\
&=&v(\lim\limits_{t\rightarrow T-}J(x,t),T)+x\lim\limits_{t\Rightarrow T-}J(x,t)\\
&=&\psi(\varphi'(x))+x\varphi'(x)\\
&=&\varphi(x),\quad x\geq 0,
\end{eqnarray*}
so $\widehat{V}$ meets the boundary and terminal conditions in \eqref{V}.

Secondly, due to \eqref{Vx}, \eqref{Vxx} and \eqref{Vt},
{\small
\begin{eqnarray}                                                  \label{LV=lv}
\Big(-\widehat{V}_t+\frac{a^2}{2}\frac{\widehat{V}^2_x}{\widehat{V}_{xx}}-rx\widehat{V}_x+\beta\widehat{V}\Big)(x,t) =\Big(-v_t-\frac{a^2}{2}y^2v_{yy}-(\beta-r)yv_y+\beta v\Big)(J(x,t),t) \geq0.
\end{eqnarray}
}
On the other hand,
\begin{eqnarray*}
\begin{array}{l}
\quad v(y,t)\geq\psi(y),\;\forall\;y>0\\
\Rightarrow\min\limits_{y>0}(v(y,t)+xy)\geq\min\limits_{y>0}(\psi(y)+xy), \;\forall\;x\geq0\\
\Rightarrow \widehat{V}(x,t)\geq \varphi(x), \;\forall\;x\geq0.
\end{array}
\end{eqnarray*}
Hence,
\[
\min\Big\{-\widehat{V}_t+\frac{a^2}{2}\frac{\widehat{V}^2_x}{\widehat{V}_{xx}}-rx\widehat{V}_x+\beta\widehat{V},\widehat{V}-\varphi\Big\}\geq0\quad \hbox{in}\quad Q_x.
\]
Now, we prove that
\begin{eqnarray}                                                                         \label{V=0}
\widehat{V}(x,t)>\varphi(x)\Rightarrow \Big(-\widehat{V}_t+\frac{a^2}{2}\frac{\widehat{V}^2_x}{\widehat{V}_{xx}}-rx\widehat{V}_x+\beta\widehat{V}\Big)(x,t)=0.
\end{eqnarray}
Before that we first claim
\begin{eqnarray}                                        \label{V>varphi}
\widehat{V}(x,t)>\varphi(x)\Rightarrow v(J(x,t),t)>\psi(J(x,t)),
\end{eqnarray}
Let $y=J(x,t)$, if $v(y,t)=\psi(y)$ for $y\geq k$, then
\begin{eqnarray*}
\begin{array}{l}
\quad v(y,t)=\psi(y)=\frac{1}{\gamma}K^\gamma\\
\Rightarrow v_y(y,t)=0\\
\Rightarrow x=0\\
\Rightarrow \widehat{V}(x,t)=\frac{1}{\gamma}K^\gamma=\varphi(x),
\end{array}
\end{eqnarray*}
And if $v(y,t)=\psi(y)$ for $y<k$, then
\begin{eqnarray*}
\begin{array}{l}
\quad v(y,t)=\psi(y)=\frac{1-\gamma}{\gamma}y^{\frac{\gamma}{\gamma-1}}+(K-b)y\\
\Rightarrow x=-v_y(y,t)=y^{\frac{1}{\gamma-1}}-(K-b)\\
\Rightarrow \widehat{V}(x,t)=v(y,t)+xy =\frac{1-\gamma}{\gamma}y^{\frac{\gamma}{\gamma-1}}+(K-b)y+(y^{\frac{1}{\gamma-1}}-(K-b))y=\varphi(x),
\end{array}
\end{eqnarray*}
Hence, \eqref{V>varphi} is true.

Combine with \eqref{V>varphi}, the variational inequality in \eqref{v} and \eqref{LV=lv} yields
\begin{eqnarray*}
\begin{array}{l}
\quad \widehat{V}(x,t)>\varphi(x)\\
\Rightarrow v(J(x,t),t)>\psi(J(x,t))\\
\Rightarrow \Big(-v_t-\frac{a^2}{2}y^2v_{yy}-(\beta-r)yv_y+\beta v\Big)(J(x,t),t)=0\\
\Rightarrow \Big(-\widehat{V}_t+\frac{a^2}{2}\frac{\widehat{V}^2_x}{\widehat{V}_{xx}}-rx\widehat{V}_x+\beta\widehat{V}\Big) (x,t)=0.
\end{array}
\end{eqnarray*}
Therefore, $\widehat{V}(x,t)$ satisfies the variational inequality in $\eqref{V}$.
So far, we have proved $\widehat{V}(x,t)$ is the strong solution of \eqref{V}.
\end{proof}
\bigskip

Now, we discuss the free boundary of \eqref{V}. Define
\begin{eqnarray*}
{\cal \varepsilon R}_x &=& \{\widehat{V}=\varphi\},\quad\hbox{exercise region}, \\
\quad{\cal CR}_x &=& \{\widehat{V}>\varphi\},\quad\hbox{continuation region}.
\end{eqnarray*}
And
\begin{eqnarray*}
&&H(t)=\sup\{x\geq0|\widehat{V}(x,t)=\varphi(x)\},\quad 0<t<T,\\
&&G(t)=\inf\{x\geq0|\widehat{V}(x,t)=\varphi(x)\},\quad 0<t<T.
\end{eqnarray*}

On the two free boundaries $y=h(t)$ and $y=g(t)$,
\begin{eqnarray*}
&&v(y,t)=\frac{1-\gamma}{\gamma}y^{\frac{\gamma}{\gamma-1}}+(K-b)y, \\
&&v_y(y,t)=-y^{\frac{1}{\gamma-1}}+(K-b).
\end{eqnarray*}

Note that
\begin{eqnarray}                                               \label{x=-vy}
x=-v_y(y,t).
\end{eqnarray}
Then the corresponding two free boundaries of \eqref{V} are
\begin{eqnarray*}
&&H(t)=-v_y(h(t),t)=h(t)^{\frac{1}{\gamma-1}}-(K-b), \\
&&G(t)=-v_y(g(t),t)=g(t)^{\frac{1}{\gamma-1}}-(K-b).
\end{eqnarray*}
Moreover
\begin{eqnarray*}
&&H'(t)=\frac{1}{\gamma-1}h(t)^{\frac{1}{\gamma-1}-1}h'(t)\geq0, \\                                               &&G'(t)=\frac{1}{\gamma-1}g(t)^{\frac{1}{\gamma-1}-1}g'(t)\leq0,
\end{eqnarray*}
and
\begin{eqnarray*}
&&H(T-)=h(T-)^{\frac{1}{\gamma-1}}-(K-b), \\
&&G(T-)=g(T-)^{\frac{1}{\gamma-1}}-(K-b).
\end{eqnarray*}

On the other hand, by \eqref{Vx} and \eqref{Jx0},
\[
\widehat{V}_x(0,t)=J(0,t)=(v(\cdot,t))^{-1}(0)=f(t).
\]
All this leads up to

\begin{theorem}\label{thm:case}
The two free boundaries of \eqref{V} satisfy $H(t),\;G(t)\in C^\infty(0,T)$ and
$H'(t)\geq0,\;G'(t)\leq0$, $\widehat{V}_x(0,t)=f(t)$.
 Moreover, they have the following classification.

\noindent {\bf Case I: $\beta\geq\frac{a^2}{2}\frac{\gamma}{1-\gamma}+r\gamma$, $\Psi (k)\geq 0$.}
\[
H(t)\equiv +\infty ,\quad k^{\frac{1}{\gamma-1}}-(K-b)=G(T-)\leq G(t),
\]
i.e.
\[
H(t)\equiv +\infty ,\quad \widehat{x}=G(T-)\leq G(t),
\]
see Fig 5.1.

\noindent {\bf Case II: $\beta\geq\frac{a^2}{2}\frac{\gamma}{1-\gamma}+r\gamma$, $\Psi (k)< 0$.}

If $\beta>\frac{a^2}{2}\frac{\gamma}{1-\gamma}+r\gamma$,
\[
H(t)\equiv+\infty,\quad y_T=\Big(\frac{-r(K-b)}{\frac{\beta-r\gamma}{\gamma}-\frac{a^2}{2}\frac{1}{1-\gamma}}\Big)-(K-b)
=G(T-)<G(t),
\]
see Fig 5.2.

If $\beta=\frac{a^2}{2}\frac{\gamma}{1-\gamma}+r\gamma$,
\[
{\cal \varepsilon R}_x=\emptyset,
\]
see Fig 5.4.

\noindent {\bf Case III: $\beta<\frac{a^2}{2}\frac{\gamma}{1-\gamma}+r\gamma$, $\Psi (k)> 0$.}
\[
k^{\frac{1}{\gamma-1}}-(K-b)=G(T-)\leq G(t)\leq H(t)\leq
H(T-)=\Big(\frac{-r(K-b)}{\frac{\beta-r\gamma}{\gamma}-\frac{a^2}{2}\frac{1}{1-\gamma}}\Big)-(K-b),
\]
see Fig 5.3.

\noindent {\bf Case IV: $\beta<\frac{a^2}{2}\frac{\gamma}{1-\gamma}+r\gamma$, $\Psi (k)\leq 0$.}
\[
{\cal \varepsilon R}_x=\emptyset,
\]
see Fig 5.4.
\end{theorem}

\begin{center}
  \begin{picture}(0,0)
    \put(-180,-90){\vector(1,0){150}}
    \put(-170,-100){\vector(0,1){95}}
    \put(-190,-20){$T$}
    \put(-170,-20){\line(1,0){140}}
    \put(-140,-15){$\widehat{x}$}
    \put(-130,-20){\circle*{2}}
    \put(-130,-70){${\cal CR}_x$}
    \put(-80,-70){${\cal \varepsilon R}_x$}
    \qbezier(-90,-90)(-100,-30)(-130,-20)
    \put(-25,-102){$x$}
    \put(-100,-40){$G(t)$}
    \put(-180,-120){Fig 5.1. \;$\beta\geq\frac{a^2}{2}\frac{\gamma}{1-\gamma}+r\gamma$, $\Psi (k)\geq 0$.}
    \end{picture}
  \begin{picture}(0,-200)

    \put(50,-90){\vector(1,0){150}}
    \put(60,-90){\vector(0,1){95}}
    \put(40,-20){$T$}
    \put(60,-20){\line(1,0){140}}
    \put(100,-70){${\cal CR}_x$}
    \put(150,-70){${\cal \varepsilon R}_x$}
    \qbezier(140,-90)(130,-30)(120,-20)
    \put(205,-102){$x$}
    \put(100,-15){$\widehat{x}$}
    \put(100,-20){\circle*{2}}
    \put(130,-40){$G(t)$}
    \put(50,-120){Fig 5.2. \;$\beta>\frac{a^2}{2}\frac{\gamma}{1-\gamma}+r\gamma$, $\Psi (k)< 0$.}
  \end{picture}
\end{center}
\bigskip
\vspace{4cm}

\begin{center}                                                                                                      \begin{picture}(0,0)
    \put(-180,-90){\vector(1,0){150}}
    \put(-170,-100){\vector(0,1){95}}
    \put(-190,-20){$T$}
    \put(-170,-20){\line(1,0){140}}
    \put(-130,-15){$\widehat{x}$}
    \put(-130,-20){\circle*{2}}
    \put(-40,-20){\circle*{2}}
    \put(-95,-50){${\cal \varepsilon R}_x$}
    \put(-50,-70){${\cal CR}_x$}
    \put(-130,-70){${\cal CR}_x$}
    \qbezier(-90,-90)(-100,-30)(-130,-20)
    \qbezier(-80,-90)(-70,-30)(-40,-20)
    \put(-25,-102){$x$}
    \put(-130,-40){$G(t)$}
    \put(-50,-40){$H(t)$}
    \put(-180,-120){Fig 5.3. \;$\beta<\frac{a^2}{2}\frac{\gamma}{1-\gamma}+r\gamma$, $\Psi (k)> 0$.}
  \end{picture}
\begin{picture}(0,0)
    \put(50,-90){\vector(1,0){150}}
    \put(60,-100){\vector(0,1){95}}
    \put(40,-20){$T$}
    \put(60,-20){\line(1,0){140}}
    \put(100,-50){${\cal CR}_x$}
    \put(205,-102){$x$}
   \put(50,-120){Fig 5.4. \;$\beta<\frac{a^2}{2}\frac{\gamma}{1-\gamma}+r\gamma$, $\Psi (k)\leq 0$,}
    \put(50,-140){\quad\quad\quad or $\beta=\frac{a^2}{2}\frac{\gamma}{1-\gamma}+r\gamma$, $\Psi (k)< 0$.}
  \end{picture}
\end{center}
\bigskip
\vspace{5cm}

\section{Conclusions}
\setcounter{equation}{0}

 In this paper we presented a new method to study the free boundaries while the exercise region is not connected (see \eqref{h(t)}-\eqref{f(t)} and Lemma \ref{lem:inCR})
so that we can shed light on the behaviors of the free boundaries  for a fully nonlinear variational inequality without any restrictions on parameters (see Figure 5.1-5.4.). The financial meaning is that if at time $t$, investor's wealth $x$ is located in
${\cal C R}_x$, then he should continue to invest;  and if investor's wealth $x$ is located in
${\cal \varepsilon R}_x$, then he should stop to investing.

%\bibliography{refs}

%\def\cprime{$'$}

\end{document}